\documentclass[12pt,twoside]{article}

\usepackage{amsfonts}
\usepackage{amsmath}
\usepackage{graphicx}

\date{}

\begin{document}

\centerline{}

\centerline{}

\centerline {\Large{\bf On the output stabilizability of the diffusion equation}}

\centerline{}

\centerline{\bf {Faouzi Haddouchi}}

\centerline{}

\centerline{Department of Physics, University of Sciences and Technology}

\centerline{El M'naouar, BP 1505, Oran, 31000, Algeria}

\centerline{fhaddouchi@gmail.com}

\newtheorem{Theorem}{\quad Theorem}[section]
\newtheorem{Definition}[Theorem]{\quad Definition}
\newtheorem{Corollary}[Theorem]{\quad Corollary}
\newtheorem{Proposition}[Theorem]{Proposition}
\newtheorem{Lemma}[Theorem]{\quad Lemma}
\newtheorem{Example}[Theorem]{\quad Example}
\newenvironment{proof}[1][Proof]{\noindent\textbf{#1:} }

\begin{abstract}
This note is devoted to study the output stabilizability
of a simplified and a one-dimensional diffusion equation. Necessary and sufficient conditions for the system to be output
stabilizable will be given. These conditions are given in terms of the eigenvalues of
the infinitesimal generator and the Fourier coefficients of input and output
operators.

\end{abstract}



\section{Introduction}

In this note, we consider the output stabilizability of the diffusion
equation on the interval $\left( 0,1\right) $:

\begin{equation}
\left\{
\begin{array}{c}
\frac{\partial z}{\partial t}=\frac{\partial ^{2}z}{\partial \xi ^{2}}%
-\alpha \frac{\partial z}{\partial \xi }+kz+b(\xi )u(t) \\
z(\xi ,0)=z_{0}(\xi ) \\
z(0,t)=z(1,t)=0%
\end{array}%
\right.
\end{equation}%
where $b\in L^{2}\left( 0,1\right) $, $\alpha >0$ and $k>0$.

With the output function given by
\begin{equation}
y(t)=\int_{_{0}}^{^{1}}\exp (-\alpha \xi )c(\xi )z(\xi ,t)d\xi.
\end{equation}%

Take $H=L^{2}\left( 0,1\right) $ to be the Hilbert space with the weighted
inner product%
\begin{equation}
\langle f,g\rangle =\int_{_{0}}^{^{1}}\exp (-\alpha \xi )f(\xi )g(\xi )d\xi .
\end{equation}%
The system (1),(2) can be rewritten in the abstract form with state space $H$%

\begin{equation}
\overset{.}{x}\left( t\right) =Ax(t)+{B}u(t),\text{ }x(0)=x_{_{0}}\
\end{equation}%
where $B=b$, $\ y(t)=\langle \ c,z(.,t)\rangle _{_{H}}=Cz(.,t)$.
\begin{equation}
A=A_{0}+kI,\text{ and \ }A_{0}h=\frac{d^{2}h}{d\xi ^{2}}-\alpha \frac{dh}{%
d\xi }
\end{equation}%
for $h$ in the domain of $A_{0}$ given by%
\begin{equation}
D\left( A_{0}\right) =\left\{
\begin{array}{c}
h:\text{ }h,\frac{dh}{d\xi }\text{{\ are \ absolutely \ continuous }} \\
\text{and\ }\frac{d^{2}h}{d\xi ^{2}}\in H,\text{ }h\left( 0\right) =h\left(
1\right) =0%
\end{array}%
\right\}
\end{equation}%
It is no hard to show that $A$ is self-adjoint with eigenvalues $\lambda
_{n}=-\frac{\alpha ^{2}}{4}-n^{2}\pi ^{2}+k$ and normalized eigenvectors
${\small \phi }_{n}{\small (\xi )=}\sqrt{2}\exp {\small (\alpha \xi /2)}\sin
{\small (n\pi \xi )}$, $n\in \mathbb{N}$, which form an orthonormal basis
for $L^{2}\left( 0,1\right) $.

A focus of this paper is to give a criterion for the output stabilization by
a linear bounded feedback $u=Fx$, $F\in L(H,\mathbb{R})$. The motivation for
considering this class of systems is given by the work of
[2], that gave a result on state stabilizability for a class of
distributed parameter systems.

The paper is structured as follows. In second section, we shall review some
well-known concepts of approximate controllability, state and output
stabilizability for infinite dimensional systems defined in Hilbert spaces.

The third section deals with controllability and stabilization for the class
of systems studied here. A fully explicit description of the controllable and
uncontrollable subspaces for this class of systems is given in section 3. We also give
a criterion for output stabilizability. Finally, we shall conclude the paper
with some examples.

\section{Preliminary Notes}

In the beginning of this section let us recall some definitions.
Consider the abstract system (S) with the state given by
\begin{equation}
\overset{.}{x}\left( t\right) =Ax(t)+Bu(t),\text{ }x(0)=x_{_{0}}\
\end{equation}%
and the output given by
\begin{equation}
y(t)=Cx(t)
\end{equation}

with the following hypothesis:

$(i)$ $x(t)\in H$ (the state space), $u(t)\in U$ (the input space) and $%
y(t)\in Y$ (the output space),\ where $H$, $U$ and $Y$ are always intended
infinite dimensional Hilbert spaces unless otherwise stated;

$(ii)$ $B$ and $C$ are linear and continuous operators,\ i.e., $B\in L(U,H)$%
, $C\in L(H,Y)$;

$(iii)$ The operator $A\ $is an infinitesimal generator of a C$%
_{_{0}}$-semigroup$\ S(t)$ on the state space $H$. As usual $u$, $x$, $y$
represent respectively the input, state and output of the system $(7)$ and $%
(8)$.

\begin{Definition}

The system $(7)$ (or the pair$\ (A$,$B)$) is approximately controllable
if $\ \mathit{N}=\left\{ 0\right\}$.

Where $\mathit{N}=\underset{{\small t\geq 0}}{\bigcap }\ker B^{\ast }S^{\ast
}\left( t\right) \ $.

$ \mathit{L}=\mathit{N}^{\perp }$ and$\ \mathit{N\ }$are called, the controllable and
uncontrollable subspaces of the system$\ (7)$, respectively.

\end{Definition}

Following [6], we can decompose the state space$\ H\mathit{\ }$as%
$\ \mathit{L}\oplus \mathit{N\ }$and then $A$, $B\ $and $C$ are represented
by the operators matrix%

\begin{equation}
A=%
\begin{pmatrix}
A_{_{11}} & 0 \\
0 & A_{_{22}}%
\end{pmatrix}%
,B=%
\begin{pmatrix}
B_{_{1}} \\
0%
\end{pmatrix}%
,C=%
\begin{bmatrix}
C_{_{1}} & C_{_{2}}%
\end{bmatrix}%
.
\end{equation}

Using these operators, we arrive at the split case:%
\begin{equation}
\left\{
\begin{array}{l}
\overset{.}{x_{_{1}}}=A_{_{11}}x_{_{1}}+B_{_{1}}u\quad \\
\overset{.}{x_{_{2}}}=A_{_{22}}x_{_{2}}\quad \quad \quad \quad \\
y=y_{_{1}}+y_{_{2}}\quad%
\end{array}%
\right.
\end{equation}

\-where $y_{_{i}}=C_{_{i}}x_{_{i}}$ , for $i=1,2$.

\begin{Definition}
The pair$\ (A,B)$ is called (exponentially )
stabilizable if there is an $\ F\in L(H,U)$ such that the semigroup $%
S_{_{A+BF}}\left( t\right) $ is (exponentially) asymptotically stable.

Where $%
S_{_{A+BF}}\left( t\right) $ is the semigroup generated by $A+BF$.
\end{Definition}

It follows immediately that if the control is given by the feedback
$u=Fx$, for all $x_{0}\in H$ there exists positive $M$ and $\omega $ such
that

\begin{center}
$\left\Vert x\left( t\right) \right\Vert \leq M\exp (-\omega t)\ \left\Vert
x_{_{0}}\right\Vert $
\end{center}

and therefore $x\left( t\right) \rightarrow 0$, if $t\rightarrow \infty $.

\begin{Definition}

The system $(7),(8)$ is output stabilizable by a bounded feedback if there
is an $F\in L(H,U)$ such that the output $y(t)$ of the closed system
\begin{equation}
\overset{.}{x}\left( t\right) =(A+BF)x(t),\text{ }x(0)=x_{_{0}}
\end{equation}

is exponentially stable, i.e., $y(t)$ converges to zero when $t\rightarrow
\infty $ , for every $x_{0}\in H$.
\end{Definition}
See e.g.,[1],[5], [6].

\section{Main Results}

Under assumption about our system operator $A$, $A$ and $S\left( t\right) $
have the spectral decompositions

\begin{equation}
{ Ax=}\overset{\infty }{\underset{_{n=1}}{\sum }}{\lambda }_{n}%
{.E(\lambda }_{n}{)x}\text{ }\ \text{\ for}\ { x\in D(A)}
\end{equation}
\begin{equation}
{{ S(t)=}\overset{\infty }{\underset{_{n=1}}{\sum }}\exp {
(\lambda }_{n}{ t)E(\lambda }_{n}{ )}}
\end{equation}
where $E(\lambda _{_{n}})$ are the spectral projections associated with the
eigenvalues $\lambda _{_{n}}$ of $A$ and are given by

\begin{equation}
{E(\lambda }_{n}{ )=}\left\langle .\ , \phi _{n}\right\rangle
{ \phi }_{n}.
\end{equation}
Furthermore,$\ x\in H$ also has the decomposition
\begin{equation}
{x=}\overset{\infty }{\underset{_{n=1}}{\sum }}{ E(\lambda }_{n}%
{ )x}.
\end{equation}

\begin{Proposition}
The system $(4)$ (or the pair $(A,b)$) is (exponentially) stabilizable if
and only if the operator $\ A_{_{22}}$ is (exponentially) stable.
\end{Proposition}

\begin{proof}
Since $(A_{_{11}},B_{_{1}})$ is approximately controllable by construction.
Then, by [5] it follows that the pair \ $(A_{_{11}},B_{_{1}})$
is exponentially stabilizable. From [6] we can get
directly the desired result.
\end{proof}
Before we shall prove our main result, we need some technical lemmas.
\begin{Lemma}
The uncontrollable subspace $\mathit{N\ }$of the system (4) is of the
following form%
\begin{equation}
{ N=}\overline{span}\left\{ \phi _{n}\text{, }n\in J\subset \mathbb{N}%
\text{ }/\text{ \ }B^{\ast }\phi _{n}{}_{\ }=0\right\}
\end{equation}%
where $B^{\ast }=\langle $ $b,$ $.\rangle _{_{H}}$ and  $\overline{span}%
\left\{ e_{n}\text{, }n\in I\right\} $ denotes the closed subspace generated
by the vectors $e_{n}\text{, }n\in I$.
\end{Lemma}

\begin{proof}
By the definition of$\ \mathit{N\ }$ and according to [4], this subspace is closed and is invariant
for $S^{\ast }\left( t\right) =S\left( t\right) $. Then by the proof of
theorem IV.6 in [3]$,$ $\mathit{N\ }$ is of the following form
\begin{center}
$\ \mathit{N}=\underset{_{n\in J}}{\sum }E(\lambda _{_{n}})\mathit{N}$ \ and
$\ E(\lambda _{_{n}})\mathit{N}\subset \mathit{N}$ \ for all $\ n\ $ in $J$
\end{center}

where\ \ $J=\left\{ n\ /\ E(\lambda _{_{n}})\ \mathit{N}\neq \left\{
0\right\} \right\}. $ We have $B^{\ast }S^{\ast }\left( t\right) x=0$ \ if
and only if \ for all \ $t\geq 0$

\begin{center}
$\overset{\infty }{\underset{_{n=1}}{\sum }}\exp (\lambda _{_{n}}t)\langle
x,\ \phi _{_{n}}\rangle $ $\langle \ b,\phi _{_{n}}\rangle =0$
\end{center}

First let $x\in E(\lambda _{_{n_{0}}})\mathit{N}$, $x\neq 0$, for a certain $ n_{_{0}}\in J$.
Then, since $E(\lambda _{_{n_{0}}})\mathit{N}$ $\subset  \mathit{N}$, it follows from [7] that

\begin{equation}
\langle x,\ \phi _{_{n_{0}}}\rangle \langle \ b,\phi _{_{_{n_{0}}}}\rangle =0
\end{equation}

Rewriting equation (17) gives

\begin{center}
$B^{\ast }\phi _{_{n_{0}}}=0.$
\end{center}

This shows that

\begin{center}
$\mathit{N}$ $\subset \overline{span}\left\{ \phi _{n}\text{, }n\in
J\subset \mathbb{N}\text{ }/\text{ }\langle \ b,\phi _{_{n}}\rangle
_{H}{}_{\ }=0\right\} $
\end{center}

Now it remains to verify that $\phi _{_{n}}\in \mathit{N}$, where $\langle \
b,\phi _{_{n}}\rangle _{H}{}_{\ }=0\ \ $for $n\in J$. But the proof of this
part is easy and will be omitted here.
\end{proof}

Using the precise description of $\mathit{N\ }$ and the fact that $\ \mathit{%
L}=\mathit{N}^{\perp }$ one can immediately get.

\begin{Lemma}

The controllable subspace $\mathit{L}$ of the system (4) is given by
\begin{equation}
\mathit{L}=\overline{span}\left\{ \phi _{n}\text{ }/\text{ \ }\langle \
b,\phi _{_{n}}\rangle _{_{H}\ }\neq 0\right\}.
\end{equation}

\end{Lemma}

As a main result of this paper we establish the following proposition:

\begin{Proposition}

The system $(4)$ is output stabilizable if and only if

\begin{equation}
\lambda _{_{n}}<0\ \ \text{for all }\ n\text{ in }\ K,
\end{equation}

where \ \ $K=\left\{ \ n\ /\ \ \langle \ c,\phi _{_{n}}\rangle \neq 0\ \text{%
and }\langle \ b,\phi _{_{n}}\rangle =0\ \right\} $.
\end{Proposition}

\begin{proof}
From [6] we have that $A_{_{ii}}$ is the infinitesimal generator of a C$_{_{0}}$-semigroup $S_{_{i}}\left( t\right) $
on $H_{_{i}}$, for $i=1,2$. $H_{_{1}}=\mathit{L}, \ H_{_{2}}=\mathit{N}$.

Furthermore, it follows that with respect to the spectral decomposition of $%
A $ we have

\begin{center}
$S_{_{1}}(t)=\underset{_{n\in I}}{\sum }\exp (\lambda _{_{n}}t)E(\lambda
_{_{n}})\ ,\ $\ $S_{_{_{2}}}(t)=\underset{_{n\in \mathbb{N}-I}}{\sum }\exp
(\lambda _{_{n}}t)E(\lambda _{_{n}})$
\end{center}

where\ \ $I=\{n\ /\ \langle \ b,\phi _{_{n}}\rangle \neq 0\}$.

According to the proof of proposition 3.1, it follows that the output $y$ of
the system $(4)$ is exponentially stabilizable if and only if the output $%
y_{_{2}}$ is exponentially stable.

In order to study the stability of the output $y_{_{2}}\left( t\right)
=C_{_{2}}x_{_{2}}\left( t\right) $ on $\mathit{N}$, \ we again consider the
subsystem

\begin{equation}
\left\{
\begin{array}{l}
\overset{.}{x_{_{2}}}=A_{_{22}}x_{_{2}},\text{ \ }x_{_{2}}(0)=x_{_{02}} \\
y_{_{2}}=C_{_{2}}x_{_{2}}\quad%
\end{array}%
\right. \
\end{equation}

where $x(0)=x_{0}=\left[
\begin{array}{c}
x_{_{01}} \\
x_{_{02}}%
\end{array}%
\right] \in \mathit{L}\oplus \mathit{N\ .}$

The output $\ y_{_{2}}\left( t\right) =C_{_{2}}S_{_{2}}\left( t\right)
x_{_{02}}$ of the subsystem (20) is given by%
\begin{equation}
\ y_{_{2}}\left( t\right) \ =\underset{_{n\in \mathbb{N}-I}}{\sum }\exp
(\lambda _{_{n}}t)\langle x_{_{0}},\ \phi _{_{n}}\rangle \langle \ c,\phi
_{_{n}}\rangle
\end{equation}%
$\ $

Using a similar argument as above one can decompose the state space $\mathit{%
N}$ of the subsystem $(20)$ as $M\oplus W$, where $M=\underset{{\small %
t\geq 0}}{\bigcap }\ker C_{_{2}}S_{_{2}}\left( t\right) $ is the unobservable subspace of the pair $(C_{_{2}},A_{_{22}})$ and $W=M^{\bot \text{ }}$is the observable subspace of the subsystem $(20)$.

The operators$\ $\ $A_{_{22}}$ , $C_{_{2}}$ may be written in the form%
\begin{equation}
\ A_{_{_{22}}}=\left(
\begin{array}{cc}
A_{_{22}}^{^{1}} & 0 \\
0 & A_{_{22}}^{^{2}}%
\end{array}%
\right) ,C_{_{2}}=\left[
\begin{array}{cc}
0 & C_{_{2}}^{^{2}}%
\end{array}%
\right]
\end{equation}

Subsystem $(20)\ $can then be written as:
\begin{equation}
\left\{
\begin{array}{l}
\overset{.}{x}_{2}^{1}=A_{22}^{1}x_{2}^{1} \\
\overset{.}{x}_{2}^{2}=A_{22}^{2}x_{2}^{2}\ \qquad \\
\text{\thinspace }y_{2}\text{ }=C_{2}^{2}x_{2}^{2}\qquad%
\end{array}%
\right.
\end{equation}%
where $x_{_{02}}=\left[
\begin{array}{c}
x_{02}^{1} \\
x_{02}^{2}%
\end{array}%
\right] \in M\oplus W\mathit{\ }$.

The stability of the output $y_{_{2}}$on $\mathit{N}$ can then be analyzed
by studying it on the observable subspace $W$ of the subsystem $(20)$. A
similar argument as that used above can be used to show that the observable
subspace of the pair $(C_{_{2}},A_{_{22}})$ is given by%
\begin{equation}
W=\overline{span}\left\{ { \phi }_{_{n}}\text{ }/\text{ \ }\langle \
c,\phi _{_{n}}\rangle \neq 0\right\}
\end{equation}

and the output $y_{_{2}}\left( t\right)
=C_{_{2}}^{^{2}}S_{_{2}}^{^{2}}\left( t\right) x_{_{02}}^{^{2}}$ of the
subsystem $(20)$ is given by%
\begin{equation}
y_{_{2}}\left( t\right) =\underset{_{n\in K}}{\sum }\exp (\lambda
_{_{n}}t)\langle x_{_{0}},\ \phi _{_{n}}\rangle \langle \ c,\phi
_{_{n}}\rangle \ \ \ \ \ \
\end{equation}

where $K=\left\{ \ n\ /\ \ \langle \ c,\phi _{_{n}}\rangle \neq 0\ \text{and
}\langle \ b,\phi _{_{n}}\rangle =0\ \right\} ,S_{_{2}}^{^{\ i}}\left(
t\right) $ being the semigroup generated by $A_{_{22}}^{^{i}}$ for $i=1,2$.

The necessary condition is straightforward. So we concentrate on the Sufficiency. From [7] and [5] it follows
that if $\lambda _{_{n}}<0$ for all $n$ in $K$, then the output $y_{_{2}}\left( t\right) $ is exponentially stable. Hence the
output $y\left( t\right) $ of the system $(4)$ is exponentially
stabilizable.

\end{proof}

\section{Examples}
\begin{Example}
By choosing

\begin{equation}
b(\xi )=\chi _{\left[ p_{1},q_{1}\right] }\left( \xi \right) ,\ c(\xi )=\chi
_{\left[ p_{2},q_{2}\right] }\left( \xi \right),
\end{equation}

where$\ {\small \chi }_{\left[ \text{ }a,\text{ }b\right] }${\small \ }%
denotes the characteristic function of the interval$\ \left[ {\small a,}%
\text{ }{\small b}\right]$. Straightforward calculations show that

\begin{equation}
\left\{
\begin{array}{l}
{ b}_{n}\text{ }{ =}\frac{{ -2}\sqrt{2}{ \alpha }}{%
{ \alpha }^{2}{ +4n}^{2}{ \pi }^{2}}{ [}e^{\frac{%
-\alpha q_{1}}{2}}{ A}_{n,{ q}_{1}}{ -}\text{ }{ e^{%
\frac{-\alpha p_{1}}{2}}\ A}_{n,p_{1}}{ ]} \\
{ c}_{n}\text{ }{ =}\frac{{ -2}\sqrt{2}{\alpha }}{%
{ \alpha }^{2}{ +4n}^{2}{ \pi }^{2}}{ [e^{\frac{%
-\alpha q_{2}}{2}}A}_{n,{ q}_{2}}{ -}\text{ }{ e^{\frac{%
-\alpha p_{2}}{2}}\ A}_{n,p_{2}}{ ]}%
\end{array}%
\right.
\end{equation}

 where $c_{n}=\left\langle c,\phi _{n}\right\rangle $ , $%
b_{n}=\left\langle b,\phi _{n}\right\rangle $, $n\in \mathbb{N}$

$A_{n,m}=(\sin (n\pi m)+(2n\pi /\alpha )\cos (n\pi m))$.

Take $p_{1}=p_{2}=1/4, q_{1}=1/2$ and $q_{2}=3/4$. Since $(A,b)$ is controllable it is clear that the
output of the system $(4)$ is exponentially stabilizable.
\end{Example}

\begin{Example}
In this example we take $\alpha=0$ and
\begin{equation}
b(\xi )=\chi _{\left[\frac{1}{4},\frac{3}{4}\right] }\left( \xi \right) ,\ c(\xi )=\chi
_{\left[\frac{1}{4},\frac{1}{2}\right] }\left( \xi \right).
\end{equation}
Elementary calculations show then that
\begin{align}
b_{_{n}}=\langle b(\xi),\phi_{_{n}}(\xi)\rangle
_{L_{2}(\frac{1}{4}\text{ },\frac{3}{4})}\ =-\frac{2\sqrt{2}}{n\pi}\sin\left[  \frac{n\pi}{2}\right]
\sin\left[  \frac{n\pi}{4}\right]  ,\text{ }\\
 c_{_{n}}=\langle c(\xi),\phi_{_{n}}(\xi)\rangle_{L_{2}(\frac{1}{4}\text{
},\frac{1}{2})}=-\frac{2\sqrt{2}}{n\pi}\sin\left[  \frac{n\pi}{8}\right]
\sin\left[  \frac{3n\pi}{8}\right].
\end{align}

A simple calculation show that the index set $K$ takes the form%
\begin{equation}
K=~\left\{  8p+2,\text{ }8p+4,\text{ }8p+6;\text{ }p\in\mathbb{N}\right\}.
\end{equation}

Thus concerning proposition 3.4, we have that for $
k={\pi}^{2}$ the stabilizability of the output $\ {\small y(t)}=\langle{\small c(\xi)},$${\small z}\rangle_{L^{2}(0,1)}$
is achieved.

\end{Example}

\end{document}